\documentclass[reqno]{amsart}

\usepackage{amssymb}
\usepackage[mathscr]{eucal}
\usepackage{hyperref}

\theoremstyle{plain}
\newtheorem{prop}{Proposition}[section]
\newtheorem{thm}[prop]{Theorem}
\newtheorem{cor}[prop]{Corollary}
\newtheorem{lem}[prop]{Lemma}

\theoremstyle{definition}
\newtheorem{dfn}[prop]{Definition}
\newtheorem{rem}[prop]{Remark}
\newtheorem{rems}[prop]{Remarks}

\newtheorem{lab}[prop]{}

\numberwithin{equation}{section}

\newcommand{\isoto}{\overset{\sim}{\to}}

\renewcommand{\subset}{\subseteq}

\newcommand{\C}{{\mathbb{C}}}
\newcommand{\N}{{\mathbb{N}}}
\renewcommand{\P}{{\mathbb{P}}}
\newcommand{\Q}{{\mathbb{Q}}}
\newcommand{\R}{{\mathbb{R}}}
\newcommand{\bbS}{{\mathbb{S}}}
\newcommand{\T}{{\mathrm{T}}}

\newcommand{\sfS}{{\mathsf{S}}}

\DeclareMathOperator{\Gal}{Gal}
\DeclareMathOperator{\Gram}{Gram}
\DeclareMathOperator{\Hom}{Hom}

\DeclareMathOperator{\spn}{span}

\DeclareTextFontCommand{\textnf}{\normalfont}

\newcommand{\gal}{{\textrm{gal}}}

\newcommand{\du}{{\scriptscriptstyle\vee}}
\renewcommand{\emptyset}{\varnothing}
\renewcommand{\setminus}{\smallsetminus}
\renewcommand{\epsilon}{\varepsilon}
\newcommand{\ol}{\overline}
\newcommand{\plus}{{\scriptscriptstyle+}}

\newcommand{\ex}{\exists\,}

\renewcommand{\choose}[2]{\genfrac(){0pt}{}{#1}{#2}}

\begin{document}

\title
[Two remarks on sums of squares with rational coefficients]
{Two remarks on sums of squares\\with rational coefficients}

\author
 {Jose Capco}
\address
 {Research Institute for Symbolic Computation, Johannes Kepler University
 \\ Altenberger Stra\ss e 69, 4040 Linz, Austria}
\email
 {jcapco@risc.jku.at}

\author
 {Claus Scheiderer}
\address
 {Fachbereich Mathematik und Statistik, Universit\"at Konstanz,
 78457 Konstanz, Germany}
\email
 {claus.scheiderer@uni-konstanz.de}

\begin{abstract}
There exist homogeneous polynomials $f$ with $\Q$-coefficients that
are sums of squares over $\R$ but not over $\Q$. The only systematic
construction of such polynomials that is known so far uses as its key
ingredient totally imaginary number fields $K/\Q$ with specific
Galois-theoretic properties. We first show that one may relax these
properties considerably without losing the conclusion, and that this
relaxation is sharp at least in a weak sense. In the second part we
discuss the open question whether any $f$ as above necessarily has a
(non-trivial) real zero. In the minimal open cases $(3,6)$ and
$(4,4)$, we prove that all examples without a real zero are contained
in a thin subset of the boundary of the sum of squares cone.
\end{abstract}

\maketitle


\section{Introduction}

Let $f\in\Q[x_1,\dots,x_n]$ be a polynomial with rational
coefficients. Given a field extension $E/\Q$ we say that $f$ is a sum
of squares over $E$, or briefly that $f$ is \emph{$E$-sos}, if there
is a polynomial identity $f=\sum_{i=1}^rp_i^2$ with $p_1,\dots,p_r\in
E[x_1,\dots,x_n]$. Suppose that $f$ is $\R$-sos. Then does it follow
that $f$ is $\Q$-sos?
This question is certainly of theoretical interest, but is also
relevant from a practical perspective. Indeed, it is one particular
instance of the general problem of finding exact (rather than
floating-point) positivity certificates.

In general, the answer is negative, as was shown in \cite{sch:rat}.
In fact, an explicit construction was presented there of
$\Q$-polynomials $f$ that are $\R$-sos but not $\Q$-sos. From this
negative answer, a series of natural follow-up questions arises, see
Section~5 in \cite{sch:rat}. In this article we discuss two of them.
Throughout we assume that our polynomials are forms, i.e.\ they are
homogeneous.

Recall that the basic construction in \cite{sch:rat} starts out with
a totally imaginary number field $K/\Q$ of even degree $2d$ and a
linear form $l\in K[x_1,\dots,x_n]$ with sufficiently general
coefficients. The norm form $f=N_{K/\Q}(l)$ is a degree~$2d$ form
over $\Q$ and is $\R$-sos. Let $K^\gal/\Q$ be the Galois hull of
$K/\Q$, and let $G=\Gal(K^\gal/\Q)$ act on the set $X=\Hom(K,\C)$ of
(complex) places of $K$ (note $|X|=2d$). According to \cite{sch:rat},
if $G$ is sufficiently ``big'' as a subgroup of the symmetric group
$S_{2d}$, then $f$ cannot be $\Q$-sos. For example, it is enough that
$G$ is doubly transitive on $X$. In Section~2 below we show that a
condition much weaker than $2$-transitivity suffices to make the
construction work (condition $(**)$), and that this relaxation is
sharp at least in a weak sense. Moreover we present empirical data
showing that up to degree $2d=16$, all transitive group actions
satisfying this condition do actually arise from a number field
$K/\Q$ as before.

All known examples of $\R$-sos forms $f\in\Q[x_1,\dots,x_n]$ that are
not $\Q$-sos have real zeros. In fact, the proofs for the
impossibility of writing $f$ as a sum of squares over $\Q$ make
crucial use of the existence of these real zeros. Therefore it is
natural to ask if there can be any examples of such forms without any
(non-trivial) real zero. See also Question 5.1 in \cite{sch:rat}. Let
$f\in\Q[x_1,\dots,x_n]$ be a form with $\deg(f)=2d$ that is $\R$-sos
and has strictly positive values. In Section~3 below we discuss the
first open case, namely $(n,2d)=(3,6)$. We conjecture that $f$ is
$\Q$-sos in this case, and we prove this conjecture for all $f$
outside of a Zariski-thin subset of the boundary of the sum of
squares cone $\Sigma_6$. An analogous result holds in the other
``Hilbert case'' $(n,2d)=(4,4)$.

In Section~3 we use some simple facts on Gram spectrahedra of forms.
For the reader's convenience we have included a brief introduction
to Gram spectrahedra in Section~4, together with proofs or references
for the facts that we use.

We remark that Question 5.3 from \cite{sch:rat} has recently been
given a negative answer by Laplagne \cite{la}. This question was
asking whether $f$ is $\Q$-sos if it becomes $K$-sos in an odd degree
extension $K/\Q$. Laplagne shows that the answer is negative. In fact,
he constructs examples of degree~$4$ polynomials
$f\in\Q[x_1,\dots,x_4]$ that are sums of squares over
$\Q(\root^3\of2)$ but not over~$\Q$.
\smallskip

\emph{Acknowledgement}:
We are indebted to J\"urgen Kl\"uners for sharing valuable
information on the Database for Number Fields \cite{dnf} with us.


\section{Conditions on the Galois group}

\begin{lab}\label{recall}
Let $K/\Q$ be a totally imaginary number field of degree $[K:\Q]=2d
\ge4$, let $K^\gal/\Q$ be its Galois hull and $G=\Gal(K^\gal/\Q)$ the
Galois group. The group $G$ acts transitively on the set
$X=\Hom(K,\C)$ of cardinality $|X|=2d$. (This action can be
identified with the $G$-action on the roots of the minimal polynomial
of a primitive element of $K/\Q$.) We fix once and for all an
embedding $K\subset\C$ and denote complex conjugation (restricted to
$K^\gal$) by $\tau\in G$. Since $K$ has no real place, $\tau$ acts on
$X$ without fixed point, i.e.\ as a product of $d$ pairwise disjoint
transpositions.

Let $l\in K[x_1,x_2,x_3]$ be a linear form, and let $f=N_{K/\Q}(l)$
be the $K/\Q$-norm of $l$. So $f\in\Q[x_1,x_2,x_3]$, and $f$ is the
product of the $2d$ Galois conjugates of~$l$. The form $f$ is a sum
of two squares of forms over the field $\R$ of real numbers. The
following was proved in \cite{sch:rat} (Sect.~2):
\end{lab}

\begin{thm}\label{jemsthm}
Suppose that the $G$-action on $X$ is $2$-transitive, or more
generally satisfies condition $(*)$ below. Then, for $l$ a linear
form with sufficiently general coefficients, the form $f=N_{K/\Q}(l)$
fails to be $\Q$-sos.
\end{thm}

``Sufficiently general coefficients'' means that no three of the $2d$
Galois conjugates of $l$ have a common nontrivial (complex) zero. For
example, when $\alpha$ is a primitive element for $K/\Q$, the form
$l=x_1+\alpha x_2+\alpha^2x_3$ has sufficiently general coefficients
in this sense. Condition $(*)$ is the following condition on the
Galois action on~$X$:
\begin{itemize}
\item[$(*)$]
\emph{For any $x,\,y\in X$ with $x\ne y$ there exists $z\in X$ and
$\sigma\in G$ such that $x=\sigma z$ and $y=\sigma\tau z$.}
\end{itemize}
Condition $(*)$ requires, in other words, that every $2$-element
subset $\{x,y\}\subset X$ is $G$-conjugate to a subset of the form
$\{z,\tau z\}$ with $z\in X$. This is a weaker condition than
$2$-transitivity, and is in fact strictly weaker (see \cite{sch:rat}
Remark 2.9 or \ref{empdata} below).

\begin{lab}\label{charnumber}
We are going to show that Theorem \ref{jemsthm} remains true under
a condition that is still much more general than condition $(*)$. To
this end let $G$ be a finite group, and let $X$ be a transitive and
faithful $G$-set (i.e.\ only $1\in G$ acts as the identity). Let
$t\in G$ be a fixed-point-free involution, i.e.\ $t$ acts on $X$
without fixed point and satisfies $t^2=1$. This forces $|X|$ to be
even. For $x\in X$ let
\begin{align*}
M_t(x) & =\ \{y\in X\colon\ex z\in X,\ \ex g\in G\text{ such that }
  x=gz,\ y=gtz\} \\
& =\ \{gtg^{-1}x\colon g\in G\},
\end{align*}
be the ``orbit'' of $x$ under the conjugacy class of $t$ in $G$. It
is easy to see that $M_t(hx)=hM_t(x)$ for any $h\in G$, $x\in X$,
and that $x\notin M_t(x)$.
Therefore the cardinality $|M_t(x)|$ is independent of $x\in M$ and
depends only on the conjugacy class of $t$ in $G$. We write
$c(G,X,t):=|M_t(x)|$ and call this the \emph{characteristic number}
of the triple $(G,X,t)$.

Property $(*)$ above says $M_t(x)=X\setminus\{x\}$ for $x\in X$, or
in other words, $(*)$ says $c(G,X,t)=|X|-1$.
\end{lab}

\begin{dfn}
Let the finite group $G$ act transitively and faithfully on the set
$X$, and let $t\in G$ be an involution without fixed points in $X$.
We say that the triple $(G,X,t)$ has property $(**)$, if
$c(G,X,t)>\frac12|X|$.
\end{dfn}

Clearly, property $(*)$ implies property $(**)$.

We return to the setting in \ref{recall}. So let again $K\subset\C$
be a totally imaginary field extension of $\Q$, of degree
$[K:\Q]=2d\ge4$ and with Galois hull $K^\gal/\Q$. We consider the
action of $G=\Gal(E/\Q)$ on $X=\Hom(K,\C)$. Let $\tau\in G$ be
complex conjugation. Theorem \ref{jemsthm} holds when condition
$(*)$ gets replaced by the weaker condition~$(**)$:

\begin{thm}\label{suff}
Suppose that the triple $(G,X,\tau)$ satisfies condition $(**)$.
Then, for $l$ a linear form with sufficiently general coefficients,
the form $f=N_{K/\Q}(l)$ fails to be $\Q$-sos.
\end{thm}

\begin{proof}
Let $l\in K[x_1,x_2,x_3]$ be a linear form, let $l_1,\dots,l_{2d}$ be
its $G$-conjugates, and assume that no three of them have a common
nontrivial complex zero. For $i=1,\dots,2d$ let
$L_i\subset\P^2(\C)$ be the projective zero set of $l_i$. For
$i\ne j$ in $\{1,\dots,2d\}$ let $M_{ij}=L_i\cap L_j$, the
intersection point of $L_i$ and $L_j$. Let $Q:=\{M_{ij}\colon1\le
i<j\le2d\}$, and let $Q_0\subset Q$ be the subset of all real points
in $Q$. By the general position assumption, $Q_0$ consists just of
the $d$ intersection points $L_i\cap\ol L_i$ ($1\le i\le2d$).

As in \cite{sch:rat}, we can identify the $2$-element subsets of $X$
with the points in $Q$, by letting $\{i,j\}\subset X$ correspond to
$M_{ij}\in Q$. This identifies the subsets $\{x,\tau x\}$ ($x\in X$)
with the points in $Q_0$. Condition $(**)$ therefore says
that one (in fact, any) of the lines $L_1,\dots,L_{2d}$ contains at
least $d+1$ points that are $G$-conjugate to a point in $Q_0$. Assume
that $f$ is $\Q$-sos, i.e.\ $f=\sum_\nu p_\nu^2$ with forms
$p_\nu\in\Q[x_1,x_2,x_3]$. Since $f(\xi)=0$ for every $\xi\in Q_0$,
we have $p_\nu(\xi)=0$ for every $\nu$ and every $\xi\in Q_0$.
Therefore every $p_\nu$ vanishes in at least $d+1$ different points
of every line $L_j$. Since $\deg(p_\nu)=d$, this implies that the
$p_\nu$ vanish identically on each line $L_j$, implying $p_\nu=0$,
contradiction.
\end{proof}

\begin{rem}\label{weaklysharp}
In a weak sense at least, condition $(**)$ is sharp for Theorem
\ref{suff}. Indeed, $(**)$ requires $c(G,X,t)\ge d+1$. If we allow
$c(G,X,t)=d$, then Theorem \ref{suff} fails in general. An example
showing this is provided by the dihedral group $G$ of order~$8$,
acting on the vertices $X$ of a square by symmetries of the square
(so $2d=4$ here). If $t\in G$ is one of the two fixed-point-free
reflections then $M_t(x)$ consists of the two vertices adjacent to
the vertex $x\in X$, and so $c(G,X,t)=2=d$. But when $K/\Q$ is an
extension with $[K:\Q]=4$ and $\Gal(K^\gal/\Q)\cong G$,
the construction in \ref{recall} always produces forms that are
$\Q$-sos. Indeed, this is a consequence of \cite{sch:rat}
Theorem~4.1.
\end{rem}

\begin{lab}
To complete this discussion, we add some empirical observations. We
consider finite transitive (faithful) group actions $(G,X)$ up to
isomorphism, i.e.\ $G$ is a finite group that acts transitively and
faithfully on the set $X$.
The notion of isomorphism $(G_1,X_1)\to(G_2,X_2)$ is obvious. For
small degrees, the isomorphism classes of such transitive group
actions are well known, and corresponding data is both available on
the web and implemented in computer algebra systems. For example, the
transitive permutation groups of degree $\le30$ are contained in a
Magma database \cite{mag}.

Given a finite extension $K\subset\C$ of $\Q$ with Galois closure
$K^\gal\subset\C$, we'll say that $K$ realizes the transitive group
$(G,X)$ if there is an isomorphism $(G,X)\isoto(\Gal(K^\gal/\Q),\>
\Hom(K,\C))$. Given moreover an involution $t\in G$, we say that $K$
realizes the triple $(G,X,t)$ if such an isomorphism can be found
that sends $t$ to (the restriction of) complex conjugation. Of
course, the realization question for $(G,X)$ is a strong version of
the inverse Galois problem, that becomes even stronger when an
involution $t\in G$ is fixed. Since we are interested in the group
actions of totally imaginary number fields, we only consider
transitive groups $(G,X)$ that contain a fixed-point-free
(\emph{fpf}) involution.
\end{lab}

\begin{lab}\label{empdata}
The following information was obtained with the help of the Magma
computer algebra system \cite{mag}. For $2d=4$ there are 5 transitive
groups $(G,X)$ (up to isomorphism). Only the $2$-transitive groups
$S_4$ and $A_4$ satisfy condition $(**)$.

For $2d=6$ there are 16 transitive groups $(G,X)$, $11$ of which
contain a fpf involution. Two of these $11$ are $2$-transitive. Two
other groups contain a fpf involution which satisfies $(*)$, namely
the transitive groups $6\T8$ and $6\T11$. (The first of them was
already discussed in \cite{sch:rat} 2.9). No further group satisfies
$(**)$.

For $2d=8$ we have the first examples of transitive groups which
satisfy $(**)$, but not $(*)$, with respect to some involution.
Note that in general, a transitive group $(G,X)$ contains several
conjugacy classes of fpf involutions $t$, whose characteristic
numbers $c(G,X,t)$ will usually be different.

For $2d=10$ or $14$, condition $(**)$ is not more general than
$(*)$.
But in degrees $12,\,16,\,18$ and $20$ there are many
transitive groups that have at least one involution satisfying
$(**)$, but no involution satisfying $(*)$. Combined with the
remarks in \ref{realization} below, this shows that we gain quite
a~bit of new examples with Theorem \ref{suff}. A precise statistics
up to degree~20 is provided in the following table:
$$\begin{array}{r|rrrr}
n & (1) & (2) & (3) & (4) \\
\hline
 4 &    5 &  2 &   - &   - \\
 6 &   11 &  2 &   2 &   - \\
 8 &   50 &  7 &   2 &   3 \\
10 &   27 &  3 &   6 &   - \\
12 &  282 &  5 &  21 &  50 \\
14 &   44 &  2 &   9 &   - \\
16 & 1954 & 13 &  35 & 120 \\
18 &  678 &  2 &  83 & 132 \\
20 & 1020 &  4 & 126 & 197
\end{array}$$
Here row $n$ contains the numbers of (isomorphism classes of)
transitive groups $(G,X)$ with $|X|=n$ that
\begin{enumerate}
\item
contain a fpf involution,
\item
contain a fpf involution and are $2$-transitive,
\item
satisfy $(*)$ for some fpf involution but are not $2$-transitive,
\item
satisfy $(**)$ for some fpf involution, but don't satisfy $(*)$ for
any fpf involution.
\end{enumerate}
\end{lab}

\begin{lab}\label{realization}
By consulting the Kl\"uners-Malle data base for number fields
(\cite{dnf}, see also \cite{km}), one can verify that for every
transitive group $(G,X)$ of degree $|X|\le16$ and every fpf
involution $t\in G$, the triple $(G,X,t)$ is realized by some number
field $K/\Q$.
We are grateful to J\"urgen Kl\"uners for confirming to us this last
assertion. A list of Galois realizations of all triples $(G,X,t)$
with $|X|\le16$ that satisfy $(**)$, extracted from \cite{dnf}, is
available under \cite{zdata}.
\end{lab}


\section{Strictly positive forms in Hilbert's sos cones}

\begin{lab}
We consider Open Problem~5.1 from \cite{sch:rat}. Let
$f\in\Q[x]=\Q[x_1,\dots,x_n]$ be a form of degree $\deg(f)=2d$ which
is $\R$-sos, and assume that $f$ is strictly positive (i.e.\ $f(a)>0$
for $0\ne a\in\R^n$). Then, does it follow that $f$ is $\Q$-sos?

Although we expect a negative answer in general, no argument or
example is known so far to decide this question in general. Here we
are looking at the first open cases. When $n\le2$ or $2d=2$, $f$ is
clearly $\Q$-sos, and the same is true for $(n,2d)=(3,4)$ according
to \cite{sch:rat} Theorem 4.1. We therefore consider the cases
$(n,2d)=(3,6)$ or $(4,4)$. They may be called the Hilbert cases,
alluding to Hilbert's celebrated 1888 theorem \cite{hi}, by which
$(3,6)$ and $(4,4)$ are the minimal cases for which there exist
nonnegative forms (over $\R$) that are not sums of squares.

In all that follows, the two cases $(3,6)$ and $(4,4)$ are completely
parallel. For simplicity we will focus on the $(3,6)$ case, and will
point out at the end how to adapt the results to the $(4,4)$ case.
Our main result in the $(3,6)$ case is Theorem \ref{posternsextq}
below. Roughly it says that if there exists strictly positive $f$
over~$\Q$ which is $\R$-sos but not $\Q$-sos, then $f$ lies in a thin
subset of the boundary of the sos cone.
\end{lab}

\begin{lab}
Let $x=(x_1,x_2,x_3)$, and consider the polynomial ring
$A=\R[x]=\R[x_1,x_2,x_3]$ with the natural grading
$A=\bigoplus_{d\ge0}A_d$. We say that a form $f\in A$ is
\emph{strictly positive}, denoted $f>0$, if $f(a)>0$ for any
$0\ne a\in\R^3$. Recall that $A_d$ is a finite-dimensional real
vector space and thus carries a unique topology. For even $d\ge0$
let $\Sigma_d\subset A_d$ be the set of all sums of squares, a
closed convex cone of full dimension (i.e.\ with non-empty interior
relative to $A_d$).

Let $f\in\Q[x]=\Q[x_1,x_2,x_3]$ be a strictly positive form of
degree~$6$ which is $\R$-sos, i.e.\ $f\in\Q[x]_6\cap\Sigma_6$.
If $f$ lies in the interior of $\Sigma_6$ then $f$ is a sum of
squares over $\Q$ (\cite{hr} Theorem 1.2 or \cite{sch:rat} Lemma
4.6). So we assume that $f$ lies on the boundary $\partial\Sigma_6$.
The Zariski closure $\partial^a\Sigma_6$ of $\partial\Sigma_6$ is
called the \emph{algebraic boundary} of $\Sigma_6$, and is known to
be a union of two irreducible hypersurfaces inside the space $A_6$.
Namely
$$\partial^a\Sigma_6\>=\>\Delta\cup V$$
where $\Delta$ is the discriminant hypersurface, consisting of all
forms with at least one (complex) singularity, and $V$ is the Zariski
closure of the sets of all sums of three squares of forms. The degree
of $\Delta$ resp.\ $V$ is $75$ resp.\ $83200$. See \cite{bhors} for
proofs of these facts (we will not make use of the precise degrees).
\end{lab}

\begin{lab}
For basic notions from convexity we refer to standard texts like
\cite{we} or \cite{ro}.
Let $A_6^\du=\Hom(A_6,\R)$ be the dual space of the linear space
$A_6$, and let $\Sigma_6^*=\{\alpha\in A_6^\du\colon\alpha(\Sigma_6)
\ge0\}$, the dual cone of $\Sigma_6$. Given $f\in\Sigma_6$, the
\emph{normal cone} of $\Sigma_6$ at $f$ is
$$N_f\>=\>N_f(\Sigma_6)\>=\>\{\alpha\in\Sigma_6^*\colon
\alpha(f)=0\},$$
a closed convex cone contained in $\Sigma_6^*$. For
$f\in\partial\Sigma_6$ we have $N_f\ne\{0\}$. Moreover, then, $N_f$
is the (closed) convex cone generated by the extreme rays
$\R_\plus\alpha$ of $\Sigma_6^*$ satisfying $\alpha(f)=0$, according
to the Krein-Milman theorem for closed pointed convex cones.
\end{lab}

\begin{lab}\label{ualpha}
Let $f\in\partial\Sigma_6$ be strictly positive. Let us recall how
Blekherman \cite{bl} proves that $f$ is a sum of three squares of
forms. Let $\alpha\in N_f$ span an extreme ray. Since $f$ is strictly
positive, $\alpha$ cannot be evaluation in a point $u\in\R^3$.
The symmetric bilinear form
$$b_\alpha\colon A_3\times A_3\to\R,\quad(p,q)\mapsto\alpha(pq)$$
is positive semidefinite. Let $U_\alpha\subset A_3$ be its kernel, so
$$U_\alpha\>=\>\{p\in A_3\colon pA_3\subset\ker(\alpha)\}\>=\>
\{p\in A_3\colon\alpha(p^2)=0\}.$$
By \cite{bl}, Corollary 2.3 and Lemma 2.4,
the forms in $U_\alpha$ have no common (real or complex) projective
zero in $\P^2$, so the projective zero set $V(U_\alpha)$ is empty.

By assumption, $f$ has at least one sums of squares representation
$f=\sum_{i=1}^rp_i^2$ (with $p_i\in A_3$). For any such identity we
have $\alpha(p_i^2)=0$ for all~$i$, so the $p_i$ lie in $U_\alpha$.
According to \cite{bl} Theorem 2.7, the linear space $U_\alpha$ has
dimension~$3$.
Therefore $f$ can be written as a sum of 3 squares. In particular,
any strictly positive $f\in\partial\Sigma_6$ lies in the
hypersurface~$V$.

Moreover, Blekherman proves that $f$ is not a sum of two squares
(\cite{bl} Corollary~1.3). Hence for every sum of squares
representation $f=\sum_{i=1}^rp_i^2$ (with $p_i\in A_3$), the linear
span of the forms $p_1,\dots,p_r$ is equal to $U_\alpha$. By Lemma
\ref{2repsameu} (see Appendix), this implies that $f$ has essentially
only one such representation:
\end{lab}

\begin{cor}\label{unique}
Let $f\in\partial\Sigma_6$ be strictly positive. Then $f$ is a sum of
three squares, and up to orthogonal equivalence, there is only one
sum of squares representation of~$f$.
\qed
\end{cor}

In particular, the Gram spectrahedron of $f$ is reduced to a single
point.

\begin{lab}\label{gorenstein}
We keep assuming that $f\in\partial\Sigma_6$ is strictly positive,
and that $\alpha\in N_f$ spans an extreme ray. Let $U_\alpha$ be the
kernel of $b_\alpha$, as in \ref{ualpha}, and let $I\subset A$ be the
(homogeneous) ideal generated by $U_\alpha$. Since $V(I)=\emptyset$
and $\dim(U_\alpha)=3$, the ideal $I$ is a complete intersection.
Therefore, according to \cite{egh} Theorem~CB8, the graded ring $A/I$
is a $0$-dimensional Gorenstein ring with socle degree~$6$,
see also \cite{bl} Theorem 2.5. In particular, $\ker(\alpha)\subset
A_6$ is the degree~$6$ part of $I$, i.e.\ $\ker(\alpha)=I_6=
U_\alpha A_3$.
\end{lab}

\begin{cor}\label{normalcone}
For every strictly positive form $f$ in $\partial\Sigma_6$, the
normal cone $N_f(\Sigma_6)$ has dimension one, i.e.\ it is a single
ray.
\end{cor}

\begin{proof}
Let $\R_\plus\alpha$, $\R_\plus\beta$ be two extreme rays contained
in $N_f$. It suffices to show that both are equal.
By the discussion in \ref{ualpha} we have $U_\alpha=U_\beta$.
But this implies $\ker(\alpha)=\ker(\beta)$, so $\alpha$ and $\beta$
are positive scalar multiples of each other.
\end{proof}

\begin{cor}\label{tangspace}
Let $f\in\partial\Sigma_6$ be strictly positive, and assume that $f$
is a nonsingular point of the hypersurface $V$. Let $\alpha\in
A_6^\du$ span the normal cone $N_f(\Sigma_6)$. Then the kernel of
$\alpha$ coincides with the tangent space to $V$ at~$f$.
\end{cor}

\begin{proof}
Let $f=p_1^2+p_2^2+p_3^2$ be the unique sos representation of $f$.
Then $U_\alpha=\spn(p_1,p_2,p_3)$ and $\ker(\alpha)=U_\alpha A_3$
(\ref{ualpha}, \ref{unique}). Consider the morphism of algebraic
varieties (affine spaces)
$\phi\colon A_3\times A_3\times A_3\to V\subset A_6$,
$(q_1,q_2,q_3)\mapsto q_1^2+q_2^2+q_3^2$, and its tangent map at the
triple $(p_1,p_2,p_3)$. The image of this tangent map is
$p_1A_3+p_2A_3+p_3A_3=U_\alpha A_3\subset A_6$, and this subspace is
contained in $T_f(V)$. We conclude $\ker(\alpha)\subset T_f(V)$, and
equality must hold since both are codimension one subspaces of $A_6$.
\end{proof}

The following result summarizes most of what we discussed so far:

\begin{cor}\label{summary}
Let $f\in\partial\Sigma_6$ be strictly positive. Then the following
hold:
\begin{itemize}
\item[(a)]
There are $p_1,\,p_2,\,p_3\in A_3$, linearly independent, with
$f=p_1^2+p_2^2+p_3^2$, and up to orthogonal equivalence, this is the
only sum of squares representation of~$f$.
\end{itemize}
Let $U=\spn(p_1,p_2,p_3)\subset A_3$, let $I\subset A$ be the ideal
generated by $U$.
\begin{itemize}
\item[(b)]
The normal cone $N_f(\Sigma_6)$ is a ray: $N_f(\Sigma_6)=
\R_\plus\alpha$ for some~$\alpha$.
\item[(c)]
$A/I$ is a Gorenstein graded algebra with socle degree~$6$.
\item[(d)]
$I_6=\ker(\alpha)=UA_3$.
\end{itemize}
If in addition, $f$ is a nonsingular point of the hypersurface $V$,
then
\begin{itemize}
\item[(e)]
$I_6=\ker(\alpha)$ is the tangent space of $V$ at~$f$.
\qed
\end{itemize}
\end{cor}

Much of Corollary \ref{summary} was already known from \cite{bl}.
New are (b), (e) and the uniqueness part of~(a).

\begin{lab}\label{tangspaceratl}
The hypersurface $V\subset A_6$ is defined by a homogeneous
polynomial $F$ (of degree $83200$) in the $\choose82=28$
coefficients of a ternary sextic. Since (the complexification of) $V$
is the Zariski closure of all sums of three squares of cubic forms,
this hypersurface is defined over $\Q$. So the polynomial $F$ can be
taken to have coefficients in $\Q$.

Let $f\in\partial\Sigma_6$ be strictly positive, so $f\in V$, and
assume that $f$ has $\Q$-coefficients. By Corollary \ref{unique},
$f$ has a unique sos representation over $\R$. We are going to show
that this representation is defined over $\Q$, provided that $f$ is
a smooth point of the hypersurface~$V$.

Indeed, since $f$ has $\Q$-coefficients, the tangent space $T_f(V)
\subset A_6$ of $V$ at $f$ is the kernel of a linear form with
$\Q$-coefficients.
By Corollary \ref{summary}(e), we conclude that the normal cone
$N_f(\Sigma_6)$ is generated by a linear form $\alpha$ with
$\Q$-coefficients. Hence the $3$-dimensional subspace
$U_\alpha=\{p\in A_3\colon A_3p\subset\ker(\alpha)\}$ of $A_3$ has a
basis consisting of $\Q$-polynomials. According to Lemma \ref{basic}
in the Appendix, this implies that the unique matrix (or tensor) in
$\Gram(f)$ has $\Q$-coefficients. Thus, $f$ is a sum of squares
over~$\Q$. Altogether this proves:
\end{lab}

\begin{thm}\label{posternsextq}
Let $f\in\Sigma_6$ be a strictly positive form with coefficients
in $\Q$. If $f$ fails to be a sum of squares over $\Q$, then $f$ lies
in the boundary of $\Sigma_6$, and $f$ is a singular point of the
hypersurface $V$.
\qed
\end{thm}

In view of this theorem, it would be very interesting to see a
characterization of the forms in $V$ that are singular as points
of~$V$. Unfortunately we don't know how to approach this question.
A direct computation of the singularities of $V$ seems hopeless due
to the complexity of the equation of $V$ (a homogeneous polynomial
in 28~variables of degree 83200).

\begin{rems}
\hfil\smallskip

1.\
Beware that a strictly positive form $f\in\Sigma_6$ which is a sum
of three squares will not in general lie on the boundary of
$\Sigma_6$. For example, the symmetric form $f=x_1^6+x_2^6+x_3^6$
is strictly positive, and it is easily seen that $f\in\text{int}
(\Sigma_6)$.
\smallskip

2.\
There is another aspect that makes the form $f=x_1^6+x_2^6+x_3^6$
interesting. Indeed, $f$ can be written as a sum of 3 squares in more
than one way, for example
$$f\>=\>(x_1^3-2x_1x_2^2)^2+(2x_1^2x_2-x_2^3)^2+x_3^6.$$
By Corollary \ref{unique}, this directly implies that $f$ lies in the
interior of $\Sigma_6$. Using the arguments from the proof of
\ref{tangspace}, we can also conclude that $f$ is a singular point
of~$V$.
\smallskip

3.\
According to Blekherman \cite{bl}, examples of strictly positive
forms in $\partial\Sigma_6$ can be constructed as follows. Let
$p_1,\,p_2\in A_3$ be two cubics which intersect transversely in
nine projective $\R$-points. For example, we may take
$p_1=x_1(x_1^2-x_3^2)$ and $p_2=x_2(x_2^2-x_3^2)$, intersecting in
nine points with affine representatives
$$\xi_1=(1,1,1),\quad\xi_2=(-1,1,1),\quad\xi_3=(1,-1,1),\quad
\xi_4=(1,1,-1),$$
$$\xi_5=(0,1,1),\quad\xi_6=(0,1,-1),\quad\xi_7=(1,0,1),\quad
\xi_8=(1,0,-1)$$
and
$$\xi_9=(0,0,1).$$
The Cayley-Bacharach relation is the unique (up to scaling) linear
relation between the nine values $p(\xi_i)$ ($1\le i\le9$) of a
general cubic~$p$. In our example it is $\sum_{i=1}^9u_ip(\xi_i)=0$
where
$$(u_1,\dots,u_9)\>=\>(1,1,1,1,-2,-2,-2,-2,4).$$
Following \cite{bl} Theorem 6.1 we consider 9-tuples
$a=(a_1,\dots,a_9)$ of nonzero real numbers with $a_i<0$ for
precisely one index $i$, that satisfy the relation
$\sum_{i=1}^9\frac{u_i^2}{a_i}=0$. For any such tuple $a$ the linear
form
$$\alpha\colon A_6\to\R,\quad\alpha(f)\>=\>\sum_{i=1}^9a_if(\xi_i)$$
is an extreme form in $\Sigma_6^*$ and is not evaluation in a point.

In our example, $a=(1,1,1,1,4,4,4,4,-2)$ is an example of such a
tuple, giving rise to $\alpha\in\Sigma_6^*$ as above. The kernel
$U_\alpha$ of the psd bilinear form $b_\alpha$ has $\dim(U_\alpha)=3$
and is spanned by $p_1,\,p_2$ and $p_3=(3x_1^2+3x_2^2-4x_3^2)x_3$.
For any three linearly independent forms $q_1,\,q_2,\,q_3$ in
$U_\alpha=\spn(p_1,p_2,p_3)$, the sextic $f=q_1^2+q_2^2+q_3^2$ is
strictly positive and lies in $\partial\Sigma_6$.
For example,
$$f\>=\>x_1^6+x_2^6+7(x_1^4+x_2^4)x_3^2+18x_1^2x_2^2x_3^2
-23(x_1^2+x_2^2)x_3^4+16x_3^6$$
is such a sextic, obtained by taking $q_i=p_i$ ($i=1,2,3$). According
to \ref{unique}, $f=p_1^2+p_2^2+p_3^2$ is, up to orthogonal
equivalence, the only sum of squares representation of~$f$ over~$\R$.

Instead of 9 real points of intersection, the two conics $p_1$ and
$p_2$ may also intersect in 7 real and one complex conjugate pair
of points. Then the $a_i$ have to satisfy a slightly different
condition, see \cite{bl} Theorem 7.1.
\end{rems}

\begin{lab}
The results and remarks in this section all carry over to the
$(4,4)$ case, i.e.\ forms $f(x_1,x_2,x_3,x_4)$ of degree~$4$, as
follows. Put $A=\R[x_1,x_2,x_3,x_4]$, let now $\Sigma_4$ denote the
sos cone in $A_4$. The Zariski closure of $\partial\Sigma_4$ is a
union $\Delta\cup V$ of two irreducible hypersurfaces, with $\Delta$
the discriminant (of degree~$108$), and $V$ (of degree~$38475$) the
Zariski closure of the sums of $4$ squares of quadratic forms
\cite{bhors}. For $f\in\partial\Sigma_4$ a strictly positive form,
the normal cone $N_f(\Sigma_4)$ is a ray, and $f$ has a unique sos
representation, which is of length~$4$. Defining the ideal
$I\subset A$ similarly to \ref{summary}, $A/I$ is Gorenstein of socle
degree~$4$. The argument in \ref{tangspaceratl} carries over (using
$4$ instead of $3$ squares), and so the analogue of Theorem
\ref{posternsextq} holds for $\Q$-forms in $\Sigma_4$. The proofs
resp.\ references are the same as in the $(3,6)$ case.
\end{lab}


\section{Background on Gram spectrahedra}

We give a brief introduction to Gram spectrahedra of forms here,
including proofs or references for a few basic facts that we are
using in Section~3. See \cite{sch:gram} for a more detailed account
of Gram spectrahedra in general.

\begin{lab}
Let $n\in\N$ and $A=\R[x_1,\dots,x_n]=\R[x]$, considered with the
usual grading $A=\bigoplus_{d\ge0}A_d$. Let $d\ge0$ and $f\in
A_{2d}$, let $X=(x_1^d,x_1^{d-1}x_2,\dots,x_n^d)$ be the list of
monomials of degree~$d$ in some fixed order, let $N=\choose{n-1+d}d$
be the number of these monomials. The \emph{Gram spectrahedron} of
$f$ is defined to be
$$\Gram(f)\>=\>\{G\in\bbS^N\colon G\succeq0,\ X^tGX=f\},$$
where $\bbS^N$ is the space of real symmetric $N\times N$ matrices
and $G\succeq0$ means that $G$ is positive semidefinite (has
non-negative eigenvalues). So $\Gram(f)$ is an affine-linear section
of the cone $\bbS^N_\plus$ of psd symmetric matrices, and one easily
checks that $\Gram(f)$ is compact.

If $f=p_1^2+\cdots+p_r^2$ with $p_i\in A_d$, and if $u_i\in\R^N$ is
the coefficients (column) vector of $p_i$, then the matrix
$G=\sum_{i=1}^ru_iu_i^t$ lies in $\Gram(f)$. Conversely, every point
of $\Gram(f)$ arises in this way. More precisely, two sum of squares
representations $f=\sum_{i=1}^rp_i^2=\sum_{i=1}^rq_i^2$ (which we can
assume to have the same length by possibly adding zero summands to
one of them) give the same element of $\Gram(f)$ if and only if they
are \emph{orthogonally equivalent}, which means that there is an
orthogonal $r\times r$ matrix $(a_{ij})$ such that $q_i=\sum_{j=1}^r
a_{ij}p_j$ ($i=1,\dots,r$). See \cite{clr} for this fact, where Gram
spectrahedra were first introduced. In other words, the points of
$\Gram(f)$ are in natural bijection with the orthogonal equivalence
classes of sum of squares representations of $f$.

For our purposes it is more convenient to represent elements of
$\Gram(f)$ as symmetric tensors $\sum_{i=1}^rp_i\otimes q_i=
\sum_{i=1}^rq_i\otimes p_i$ with $p_i,\,q_i\in A_d$, i.e.\ as
elements of $\sfS_2A_d$, the subspace of $A_d\otimes A_d$ of
symmetric tensors. Using the multiplication (linear) map
$\mu\colon\sfS_2A_d\to A_{2d}$, $\Gram(f)$ consists of all
$\theta\in\sfS_2A_d$ with $\mu(\theta)=f$ and $\theta\succeq0$. Here
$\theta\succeq0$ means that $\theta$ can be written
$\theta=\sum_{i=1}^rp_i\otimes p_i$ with $p_i\in A_d$. See
\cite{sch:gram} for this point of view.
\end{lab}

\begin{lab}\label{dimfaces}
Let $\theta\in\Gram(f)$, say $\theta=\sum_{i=1}^rp_i\otimes p_i$.
With $\theta$ we associate the linear subspace $U_\theta:=
\spn(p_1,\dots,p_r)$ of $A_d$. The supporting face $F$ of $f$
in $\Gram(f)$ (i.e.\ the unique face that contains $f$ in its
relative interior) consists of all $\eta\in\Gram(f)$ with
$U_\eta\subset U_\theta$. It therefore corresponds to the sos
representations of
$f$ that use only polynomials from $U_\theta$. For describing the
dimension of $F$ we can assume that $p_1,\dots,p_r$ are linearly
independent. Then $\dim(F)$ is the number of linear relations between
the forms $p_ip_j$ ($1\le i\le j\le r$), i.e.
$$\dim(F)+\dim(U_\theta U_\theta)\>=\>\choose{r+1}2$$
where $U_\theta U_\theta:=\spn(p_ip_j\colon1\le i\le j\le r)$
(\cite{sch:gram} Proposition 3.6). In particular, $\theta$ is an
extreme point of $\Gram(f)$ if and only if the $\choose{r+1}2$ forms
$p_ip_j$ are linearly independent. In this case we say that
$p_1,\dots,p_r$ are \emph{quadratically independent.}
\end{lab}

\begin{lem}\label{2repsameu}
Assume that $f\in A_{2d}$ has two non-equivalent sos representations
$f=\sum_{i=1}^rp_i^2=\sum_{i=1}^rq_i^2$ with $\spn(p_1,\dots,p_r)=
\spn(q_1,\dots,q_r)=:U$. Then $f$ has another sos representation
$f=\sum_{j=1}^su_j^2$ for which $\spn(u_1,\dots,u_s)$ is a proper
subspace of~$U$.
\end{lem}

\begin{proof}
The two given representations represent two different points in the
Gram spectrahedron $\Gram(f)$, that both lie in the relative interior
of the same face $F$ of $\Gram(f)$.
Since $\Gram(f)$ is compact, $F$ has some extreme point $\theta$. If
$f=\sum_{j=1}^su_j^2$ is an sos representation that corresponds to
$\theta$, then $\spn(u_1,\dots,u_s)$ is a proper subspace of~$U$.
\end{proof}

The following lemma is used in the proof of our main result in
Section~3:

\begin{lem}\label{basic}
Let $f\in\Sigma_{2d}$ be a form with $\Q$-coefficients, and let
$\theta$ be an extreme point of $\Gram(f)$. If the space $U_\theta$
is defined over $\Q$, then the sos representation of $f$
corresponding to $\theta$ is (can be) defined over~$\Q$.
\end{lem}

That $U_\theta$ is defined over $\Q$ means that $U_\theta$ has a
linear $\R$-basis consisting of polynomials with $\Q$-coefficients.

\begin{proof}
Let $p_1,\dots,p_r$ be a basis of $U_\theta$ consisting of
$\Q$-polynomials. Since $\theta$ is an extreme point of $\Gram(f)$,
the forms $p_i$ are quadratically independent, see \ref{dimfaces}.
Hence there is a unique $\R$-linear combination $f=\sum_{i,j=1}^r
a_{ij}p_ip_j$ with $a_{ij}=a_{ji}\in\R$. The matrix $(a_{ij})$ is
positive definite, and $a_{ij}\in\Q$ by uniqueness of the linear
combination. So $f$ is $\Q$-sos.
\end{proof}


\end{document}